\documentclass[12pt]{amsart}
\usepackage{amsmath,amssymb,amsfonts,epsfig}

\newtheorem{thm}{Theorem}[section]
\newtheorem{lem}[thm]{ Lemma}
\newtheorem{cor}[thm]{ Corollary}
\theoremstyle{definition}

\newtheorem{rem}[thm]{ Remark}

\newcommand{\R}{\mathbb{R}}
\def\O{\Omega}

\begin{document}

\title{Minimal generating sets of Reidemeister moves}
\author[Michael Polyak]{Michael Polyak}

\address{Department of mathematics, Technion, Haifa 32000, Israel}
\email{polyak@math.technion.ac.il}

\begin{abstract}
It is well known that any two diagrams representing the same
oriented link are related by a finite sequence of Reidemeister moves
$\O1$, $\O2$ and $\O3$. Depending on orientations of fragments
involved in the moves, one may distinguish 4 different versions of
each of the $\O1$ and $\O2$ moves, and 8 versions of the $\O3$ move.
We introduce a minimal generating set of 4 oriented Reidemeister
moves, which includes two $\O1$ moves, one $\O2$ move, and one $\O3$
move. We then study which other sets of up to 5 oriented moves
generate all moves, and show that only few of them do. Some commonly
considered sets are shown not to be generating. An unexpected
non-equivalence of different $\O3$ moves is discussed.
\end{abstract}

\keywords{Reidemeister moves, knot and link diagrams}

\subjclass[2000]{57M25, 57M27}

\maketitle


\section{Introduction} \label{s:intro}
A standard way to describe a knot or a link in $\R^3$ is via its
{\em diagram}, i.e. a generic plane projection of the link such that
the only singularities are transversal double points, endowed with
the over/undercrossing information at each double point. Two
diagrams are equivalent if there is an orientation-preserving
diffeomorphism of the plane that carries one diagram to the other
diagram. A classical result of Reidemeister \cite{Re} states that
any two diagrams of isotopic links are related by a finite sequence
of simple moves\footnote{Our notation makes no distinction between a
move and the inverse move.} $\O1$, $\O2$, and $\O3$, shown in Figure
\ref{fig:Reidem}.

\begin{figure}[htb]
\centerline{\includegraphics[width=5in]{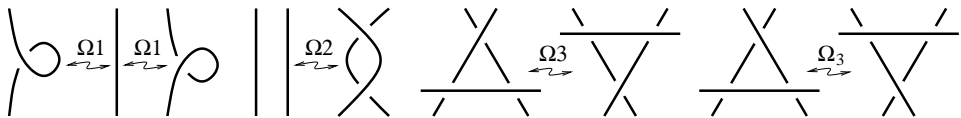}}
\caption{\label{fig:Reidem} Reidemeister moves}
\end{figure}

Here we assume that two diagrams $D$, $D'$ related by a move
coincide outside an oriented embedded disk $C\subset \R^2$ (with an
orientation of $C$ induced by the standard orientation of $\R^2$),
called the {\em changing disk}, and look as a corresponding pair
$R$, $R'$ of arc diagrams in Figure \ref{fig:Reidem} inside $C$. In
other words, there are two orientation-preserving diffeomorphisms
$f,f':C\to B^2$ of $C$ to the standard oriented 2-disk $B^2$, such
that $f_{\partial C}=f'_{\partial C}$ and $f(C\cap D)=R$, $f'(C\cap
D')=R'$.

To deal with oriented links we consider oriented diagrams. Depending
on orientations of fragments involved in the moves, one may
distinguish four different versions of each of the $\O1$ and $\O2$
moves, and eight versions of the $\O3$ move, see Figures
\ref{fig:OrientedSet1}, \ref{fig:OrientedSet2}, and
\ref{fig:OrientedSet3} respectively.

\begin{figure}[htb]
\centerline{\includegraphics[height=0.73in]{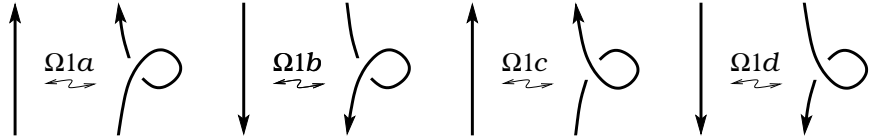}}
\caption{\label{fig:OrientedSet1} Oriented Reidemeister moves of
type 1}
\end{figure}

\begin{figure}[htb]
\centerline{\includegraphics[width=5in]{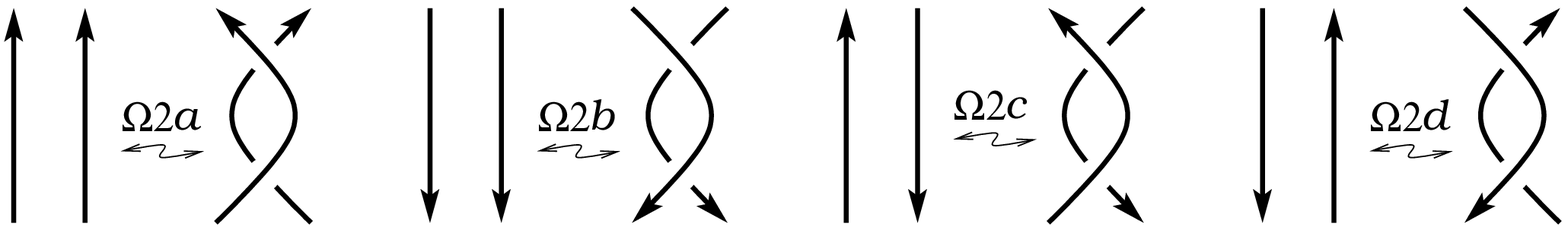}}
\caption{\label{fig:OrientedSet2} Oriented Reidemeister moves of
type 2}
\end{figure}

\begin{figure}[htb]
\centerline{\includegraphics[width=4.6in]{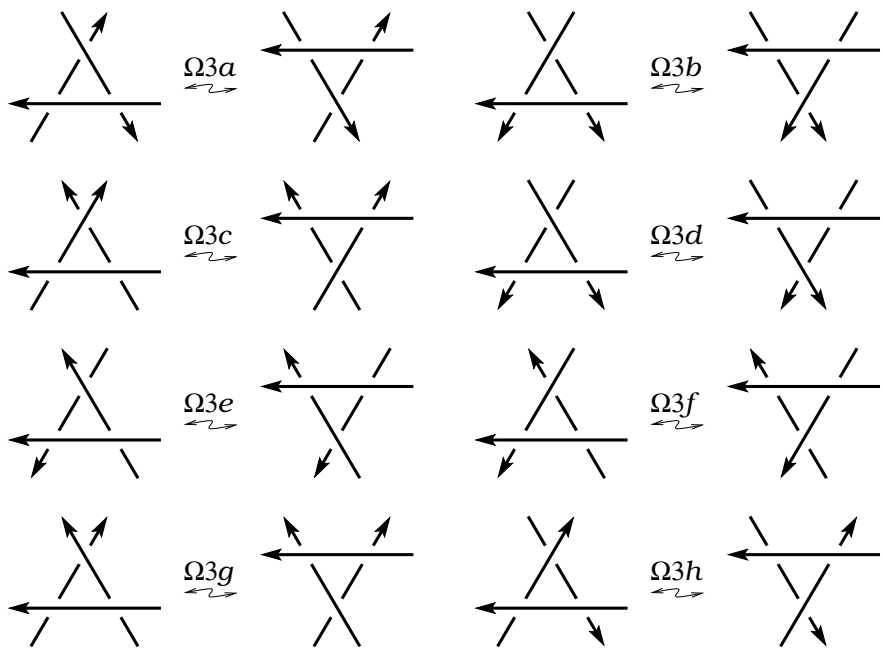}}
\caption{\label{fig:OrientedSet3} Oriented Reidemeister moves of type 3}
\end{figure}

When one checks that a certain function of knot or link diagrams
defines a link invariant, it is important to minimize the number of
moves. We will call a collection $S$ of oriented Reidemeister moves
a {\em generating set}, if any oriented Reidemeister move $\O$ may
be obtained by a finite sequence of isotopies and moves from the set
$S$ inside the changing disk of $\O$.

While some dependencies between oriented Reidemeister moves are
well-known, the standard generating sets of moves usually include
six different $\O3$ moves, see e.g. Kauffman \cite{Ka}. For sets
with a smaller number of $\O3$ moves there seems to be a number of
different, often contradictory, results. In particular, Turaev
\cite[proof of Theorem 5.4]{Tu} introduces a set of five oriented
Reidemeister moves with only one $\O3$ move. There is no proof (and
in fact we will see in Section \ref{sec:othersets} that this
particular set is not generating), with the only comment being a
reference to a figure where, unfortunately, a move $\O2$ which does
not belong to the set is used. Wu \cite{Wu} uses the same set of
moves citing \cite{Tu}, but additionally incorrectly puts the total
number of oriented $\O3$ moves at 12 (instead of 8). Kaufmann
\cite[page 90]{Ka} includes as an exercise a set of all $\O1$ and
$\O2$ moves together with two $\O3$ moves. Meyer \cite{Me} uses a
set with four $\O1$, two $\O2$, and two $\O3$ moves and states
(again without a proof) that the minimal number of needed $\O3$
moves is two. The number of $\O3$ moves used by \"Ostlund \cite{Oe}
is also two, but his classification of $\O3$ moves works only for
knots and is non-local (depending on the cyclic order of the
fragments along the knot). Series of exercises in Chmutov et al.
\cite{CDM} (unfortunately without proofs) suggest that only one
$\O3$ suffices, but this involves all $\O2$ moves. These
discrepancies are most probably caused by the fact that while many
people needed some statement of this kind, it was only an auxiliary
technical statement, a proof of which would be too long and would
take the reader away from the main subject, so only a brief comment
was usually made. We believe that it is time for a careful
treatment. In this note we introduce a simple generating set of four
Reidemeister moves, which includes two $\O1$ moves, one $\O2$ move
and one $\O3$ move:

\begin{figure}[htb]
\centerline{\includegraphics[width=5in]{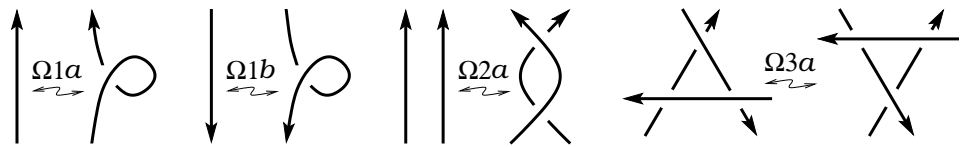}}
\caption{\label{fig:MinimalSet} A generating set of Reidemeister
moves}
\end{figure}

\begin{thm}\label{thm:main}
Let $D$ and $D'$ be two diagrams in $\R^2$, representing the same
oriented link. Then one may pass from $D$ to $D'$ by isotopy and a
finite sequence of four oriented Reidemeister moves $\O1a$,  $\O1b$,
$\O2a$, and $\O3a$, shown in Figure \ref{fig:MinimalSet}.
\end{thm}

This generating set of Reidemeister moves has the minimal number of
generators. Indeed, it is easy to show that any generating set
should contain at least one move of each of the types two and three;
Lemma \ref{lem:O1} in Section \ref{sec:othersets} implies that there
should be at least two moves of type one. Thus any generating set of
Reidemeister moves should contain at least four moves.

Our choice of the move $\O3a$ as a generator may look unusual, since
this move (called a cyclic $\O3$ move, see e.g. \cite{Ka}) is rarely
included in the list of generators, contrary to a more common move
$\O3b$, which is the standard choice motivated by the braid
theory\footnote{This is the only $\O3$ move with all three positive
crossings.}. The reason is that, unexpectedly, these moves have
different properties, as we discuss in detail in Section
\ref{sec:othersets}. Indeed, Theorem \ref{thm:othersets} below
implies that any generating set of Reidemeister moves which includes
$\O3b$ has at least five moves. If we consider sets of five
Reidemeister moves which contain $\O3b$, then it turns out that out
of all combinations of $\O1$ and $\O2$ moves, only 4 sets generate
all Reidemeister moves. The only freedom is in the choice of $\O1$
moves, while $\O2$ moves are uniquely determined:

\begin{thm}\label{thm:othersets} Let $S$ be a set of at most five
Reidemeister moves which contains only one move, $\O3b$, of type
three. The set $S$ generates all Reidemeister moves if and only if
$S$ contains $\O2c$ and $\O2d$ and contains one of the pairs
($\O1a$, $\O1b$), ($\O1a$, $\O1c$), ($\O1b$, $\O1d$), or ($\O1c$,
$\O1d$).
\end{thm}

\begin{figure}[htb]
\centerline{\includegraphics[width=5in]{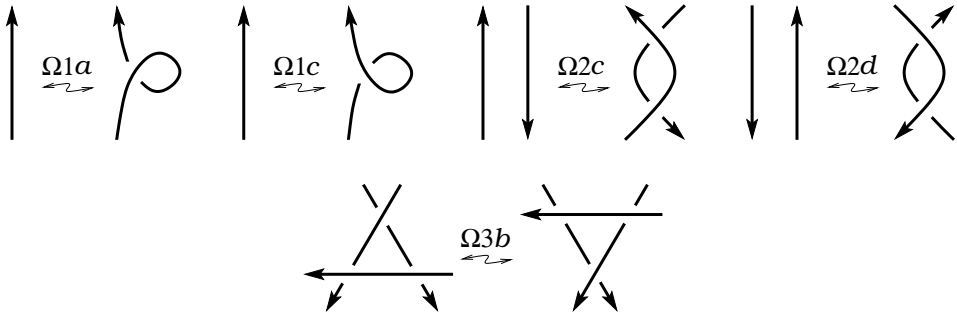}}
\caption{\label{fig:BigSet} A generating set of Reidemeister moves
containing $\O3b$}
\end{figure}

One of these generating sets is shown in Figure \ref{fig:BigSet}. It
is interesting to note that while (by Markov theorem) the set
$\O1a$, $\O1c$, $\O2a$, $\O2b$ and $\O3b$ shown in Figure
\ref{fig:NotSet} allows one to pass between any two braids with
isotopic closures, this set is not sufficient to connect any pair of
general diagrams representing the same link. This means that some
extra moves should appear in the process of transforming a general
link diagram into a closed braid. And indeed, in all known
algorithms of such a transformation the additions moves occur. For
example, in Vogel's algorithm \cite{Vo} the moves $\Omega2c$ and
$\Omega2d$ are the main steps of the algorithm.

\begin{figure}[htb]
\centerline{\includegraphics[width=5in]{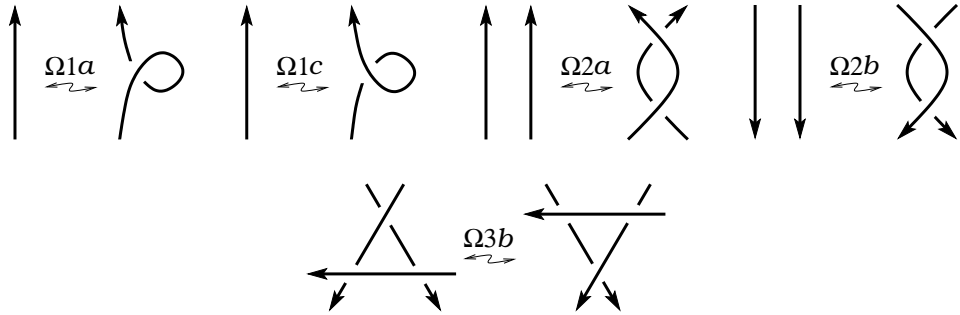}}
\caption{\label{fig:NotSet} This is not a generating set}
\end{figure}

Even more unexpected is the fact that all type one moves together
with $\O2a$, $\O2c$ (or $\O2d$) and $\O3b$ are also insufficient,
see Figure \ref{fig:NotSet2} .

\begin{figure}[htb]
\centerline{\includegraphics[width=5in]{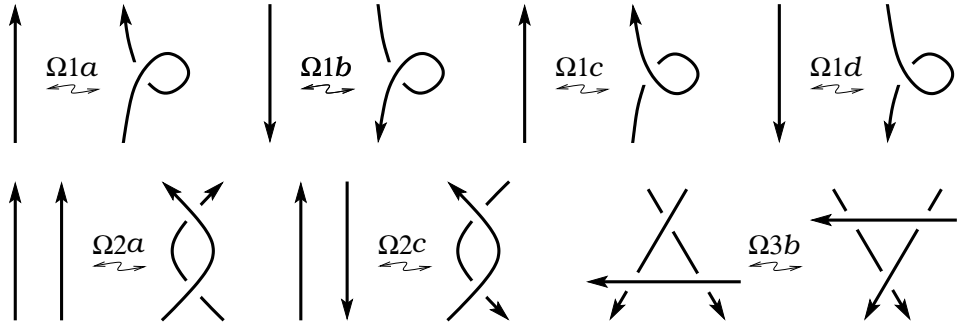}}
\caption{\label{fig:NotSet2} Another set which is not generating}
\end{figure}

\begin{rem}
For non-oriented links a sequence of Reidemeister moves can be
arranged in such a form that first a number of $\O1$  moves are
performed, then $\O2$ moves are performed, after this $\O3$ moves
are performed, and finally $\O2$ moves have to be performed again,
see \cite{Co}. It would be interesting to find such a theorem for
oriented case.
\end{rem}

All our considerations are local, and no global realization
restrictions are involved. Therefore all our results hold also for
virtual links.

Section \ref{sec:main} is dedicated to the proof of Theorem
\ref{thm:main}. In Section \ref{sec:othersets} we discuss various
generating sets which contain $\O3b$ and prove Theorem
\ref{thm:othersets}

We are grateful to O.~Viro for posing the problem and to S.~ Chmutov
for valuable discussions. The author was supported by an ISF grant
1261/05 and by the Joseph Steiner family foundation.

\section{A minimal set of oriented Reidemeister moves}
\label{sec:main}

In this section we prove Theorem \ref{thm:main} in several easy
steps. The first step is to obtain $\O2c$, $\O2d$:

\begin{lem}\label{lem:O2cd}
The move $\O2c$ may be realized by a sequence of $\O1a$, $\O2a$ and
$\O3a$ moves.  The move $\O2d$ may be realized by a sequence of
$\O1b$, $\O2a$ and $\O3a$ moves.
\end{lem}
\begin{proof}
$$\centerline{\includegraphics[width=4.5in]{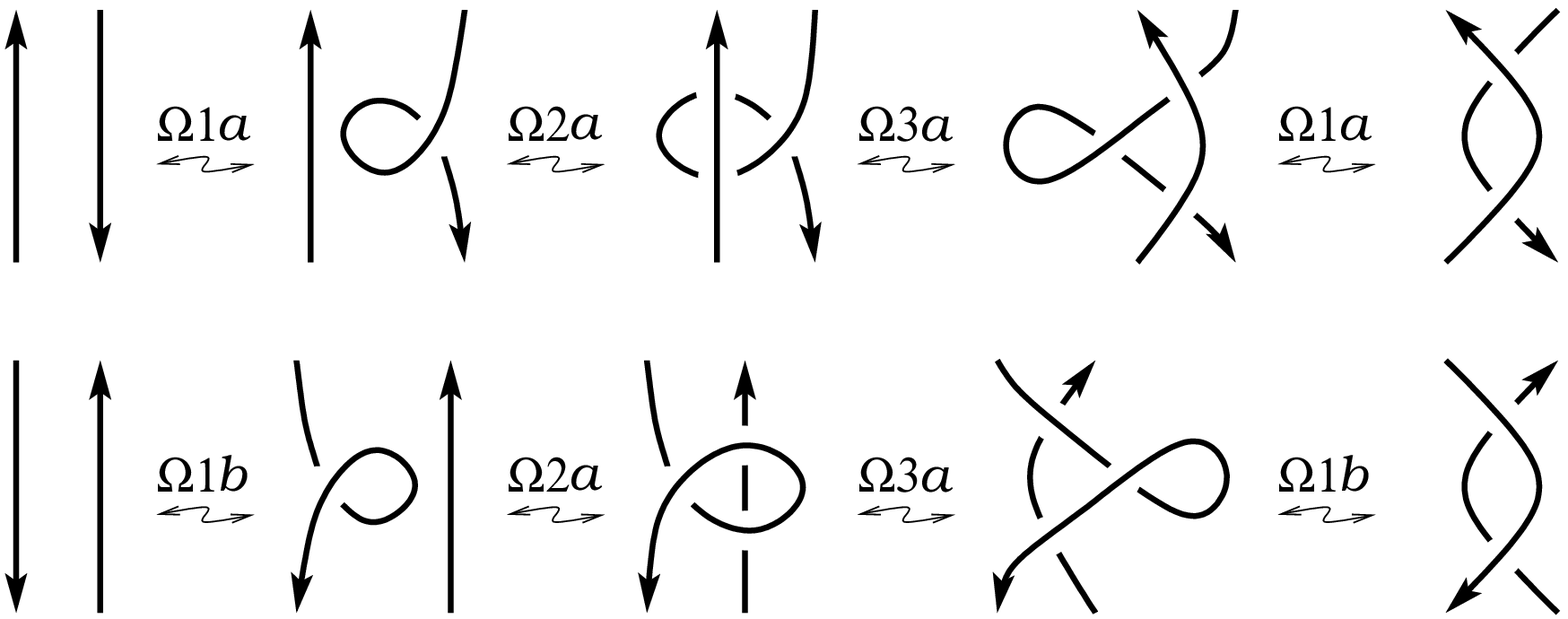}}$$
\end{proof}

Now the remaining moves of type one may be obtained as in \cite{Oe}:

\begin{lem}[\cite{Oe}]\label{lem:O1}
The move $\O1c$ may be realized by a sequence of $\O1b$ and $\O2d$
moves.  The move $\O1d$ may be realized by a sequence of $\O1a$ and
$\O2c$ moves.
\end{lem}
\begin{proof}
$$\centerline{\includegraphics[width=4.8in]{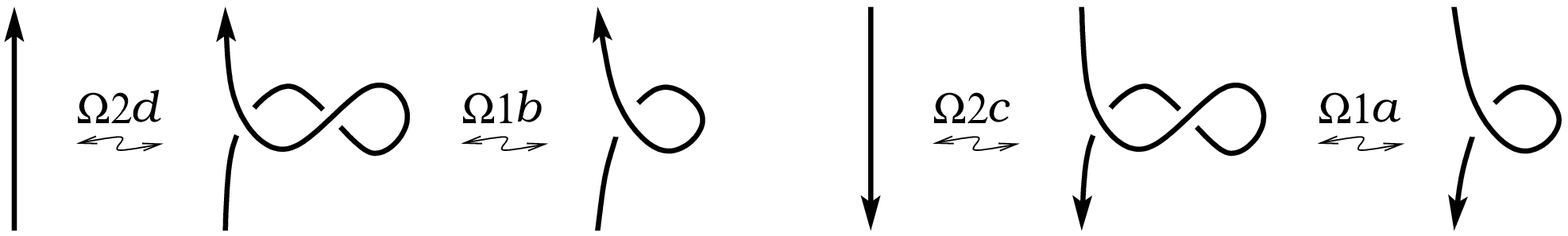}}$$
\end{proof}

This concludes the treatment of all $\O1$ and $\O2$ moves, except
for $\O2b$; we will take care of it later. Having in mind Section
\ref{sec:othersets}, where we will deal with $\O3b$ instead
of $\O3a$, we will first consider $\O3b$:

\begin{lem}\label{lem:O3b}
The move $\O3b$ may be realized by a sequence of $\O2c$, $\O2d$,
and $\O3a$ moves.
\end{lem}
\begin{proof}
$$\centerline{\includegraphics[width=5in]{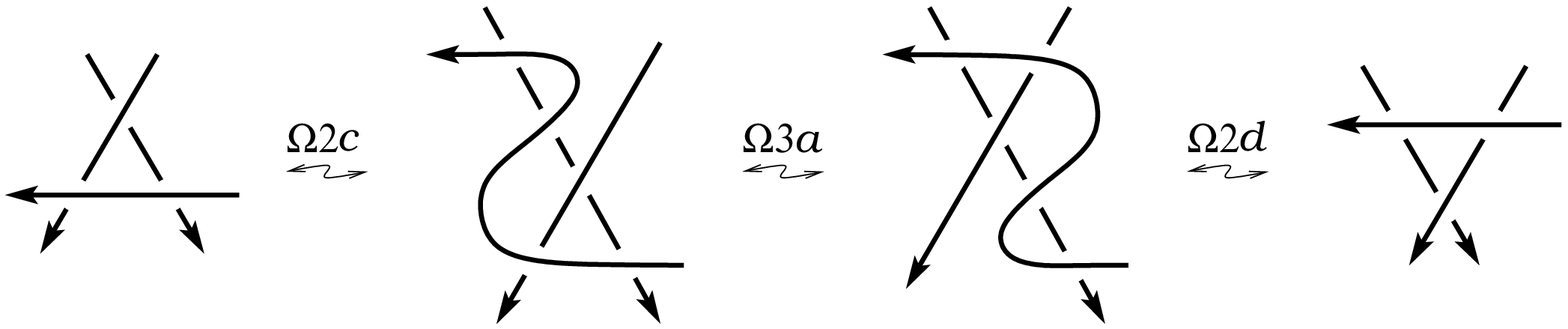}}$$
\end{proof}

To deal with $\O2b$ we will need another move of type three:
\begin{lem}
The move $\O3c$ may be realized by a sequence of $\O2c$, $\O2d$, and
$\O3a$ moves.
\end{lem}
\begin{proof}
$$\centerline{\includegraphics[width=5in]{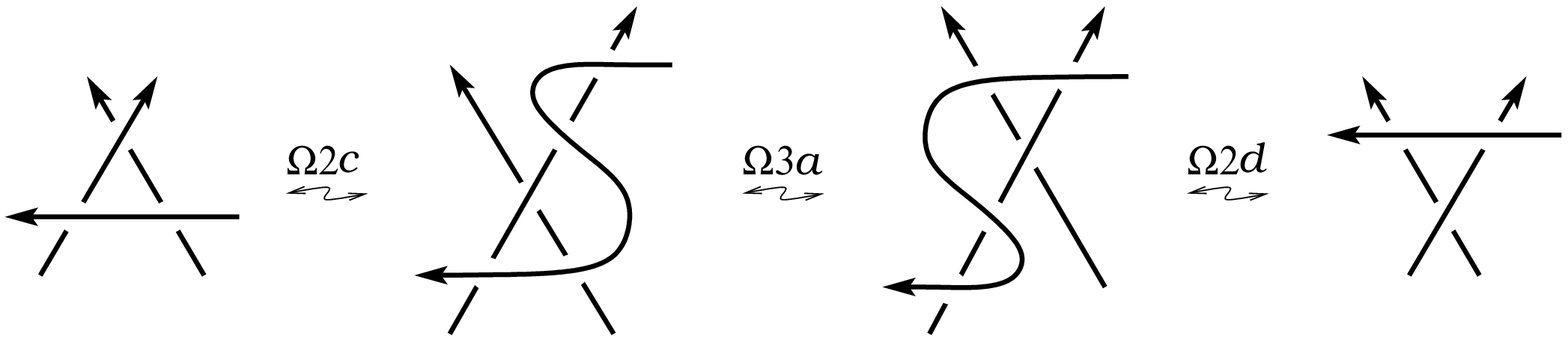}}$$
\end{proof}

At this stage we can obtain the remaining move $\O2b$ of type two:
\begin{lem}
The move $\O2b$ may be realized by a sequence of $\O1d$, $\O2c$ and
$\O3c$ moves.
\end{lem}
\begin{proof}
$$\centerline{\includegraphics[width=4.2in]{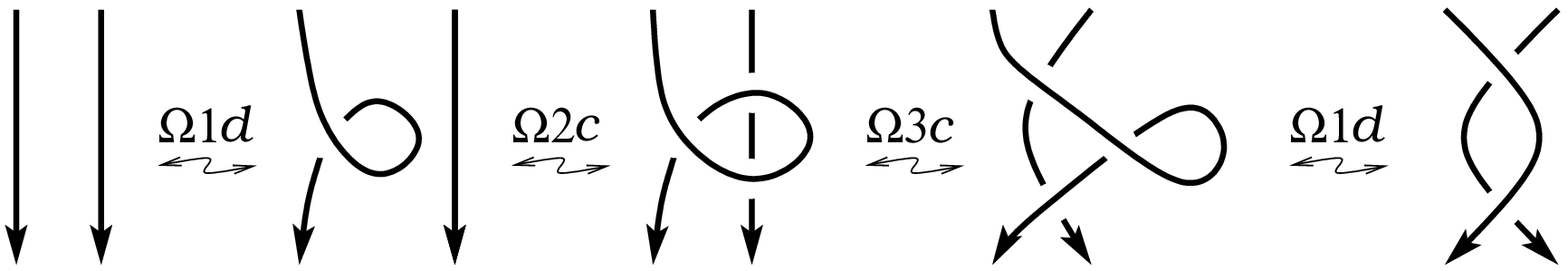}}$$
\end{proof}

To conclude the proof of Theorem \ref{thm:main}, it remains to
obtain $\O3d$ -- $\O3h$. Since by now we have in our disposal all
moves of type two, this becomes an easy exercise:
\begin{lem}\label{lem:O3}
The moves $\O3d$ -- $\O3h$ of type three may be realized by a
sequence of type two moves, $\O3a$, and $\O3b$.
\end{lem}
\begin{proof}
We realize moves $\O3d$-$\O3h$ as shown in rows 1-5 of the figure
below, using $\O3f$ to get $\O3g$, and $\O3g$ to get $\O3h$:
$$\centerline{\includegraphics[width=5in]{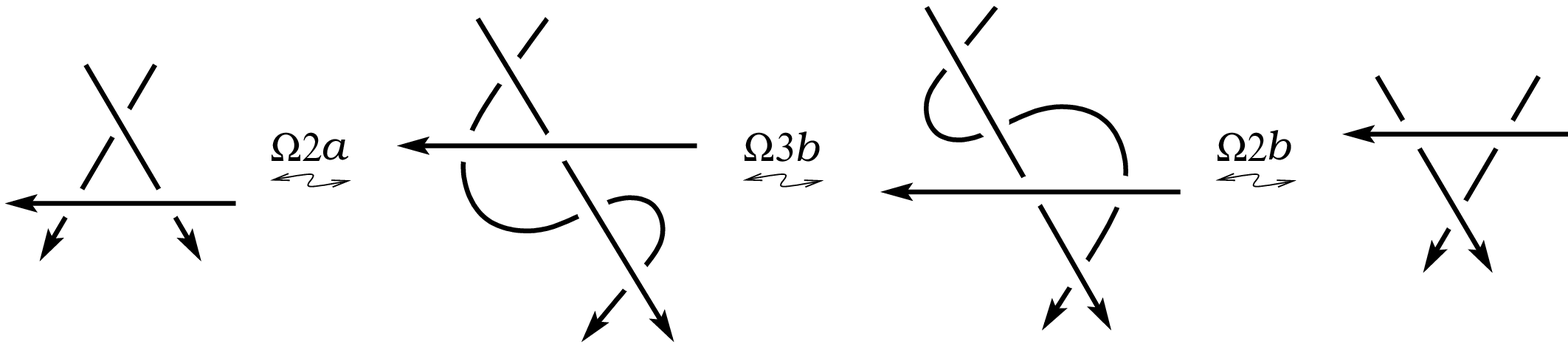}}$$
\vspace{-0.11in} \vspace{-0.11in}
$$\centerline{\includegraphics[width=5in]{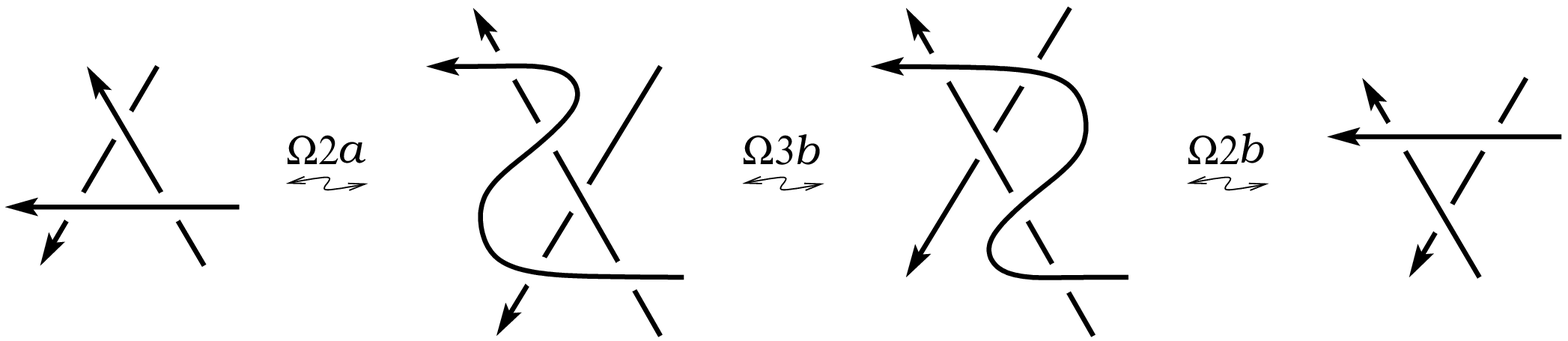}}$$
\vspace{-0.11in}
$$\centerline{\includegraphics[width=5in]{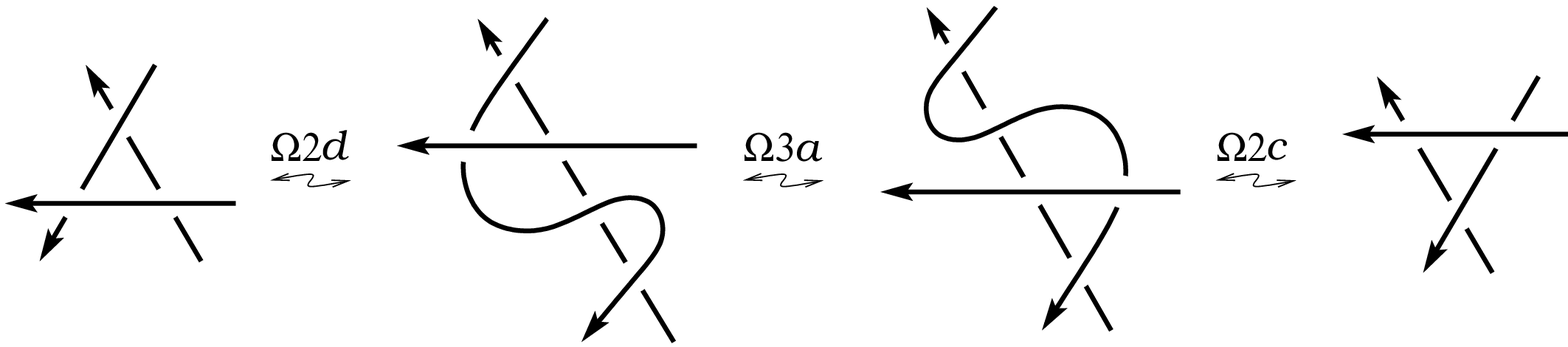}}$$
\vspace{-0.11in}
$$\centerline{\includegraphics[width=5in]{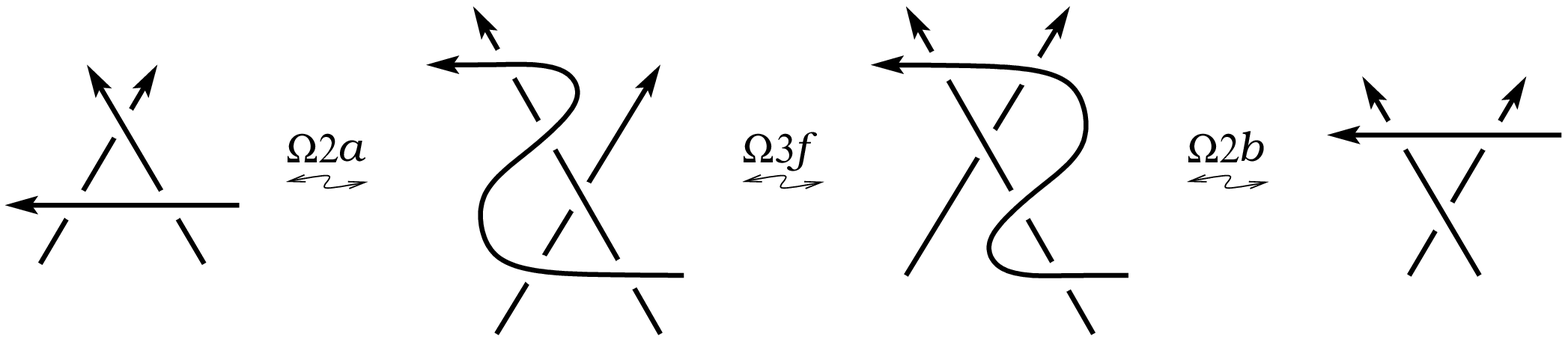}}$$
\vspace{-0.11in}
$$\centerline{\includegraphics[width=5in]{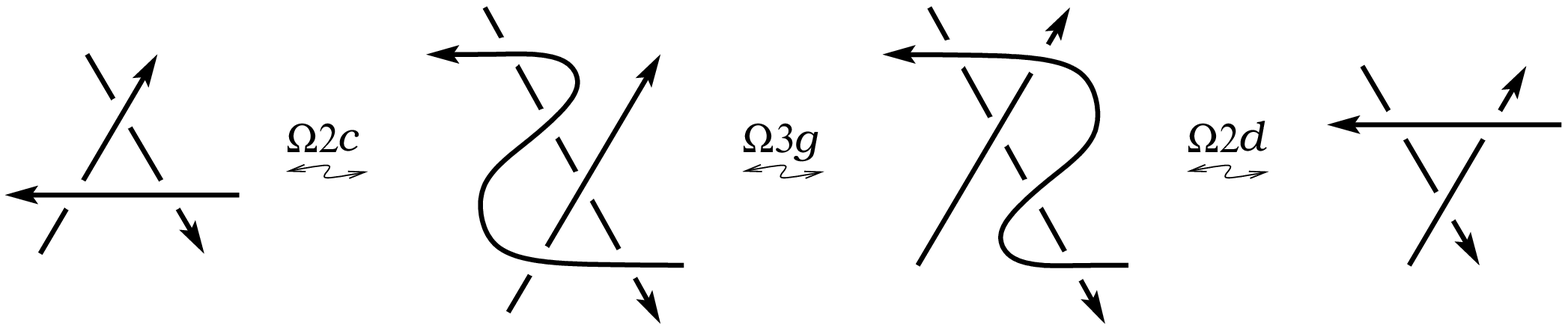}}$$
\vspace{-0.11in}
\end{proof}

\begin{rem}
There are other generating sets which include $\O3a$. In particular,
$\O1a$, $\O1b$, $\O2b$ and $\O3a$ also give a generating set. To
adapt the proof of Theorem \ref{thm:main} to this case, one needs
only a slight modification of Lemma \ref{lem:O2cd}. All other lemmas
do not change.
\end{rem}

\section{Other sets of Reidemeister moves} \label{sec:othersets}

In this section we discuss other generating sets and prove Theorem
\ref{thm:othersets}. Unexpectedly, different $\O3$ moves have
different properties as far as generating sets of Reidemeister moves
are concerned. Let us study the case of $\O3b$ in more details, due
to its importance for braid theory.

In a striking contrast to Theorem \ref{thm:main} which involves
$\O3a$, Theorem \ref{thm:othersets} implies that there does not
exist a generating set of four moves which includes $\O3b$. It is
natural to ask where does the proof in Section \ref{sec:main}
breaks down, if we attempt to replace $\O3a$ with $\O3b$.

The only difference between $\O3a$ and $\O3b$ may be pinpointed to
Lemma \ref{lem:O2cd}: it does not have an analogue with $\O3b$
replacing $\O3a$, as we will see in the proof of Lemma
\ref{lem:O2abO3b} below.

An analogue of Lemma \ref{lem:O3b} is readily shown to exist.
Indeed, $\O3a$ may be realized by a sequence of $\O2c$, $\O2d$ and
$\O3b$ moves, as illustrated below:
$$\centerline{\includegraphics[width=5in]{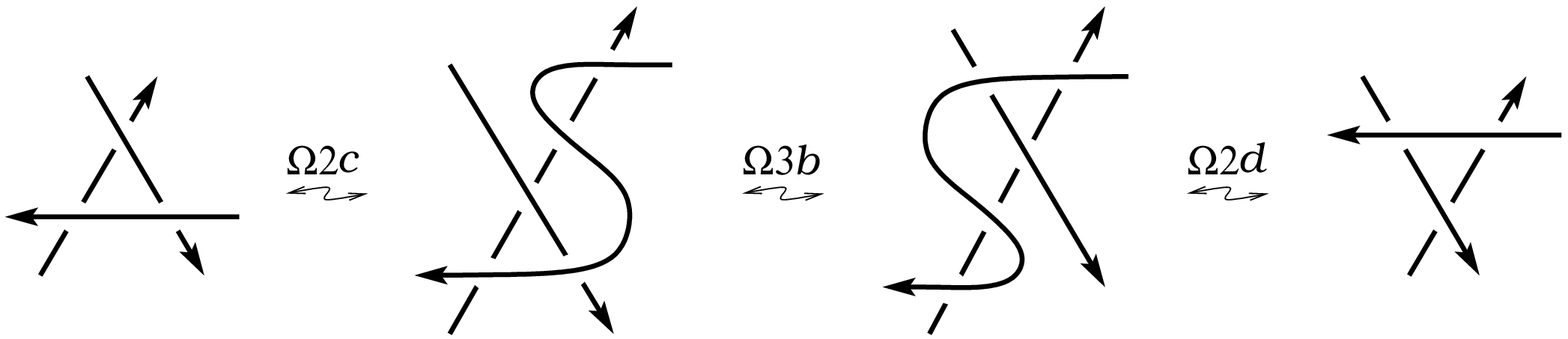}}$$
Using this fact instead of Lemma \ref{lem:O3b}, together with the
rest of Lemmas \ref{lem:O1}-\ref{lem:O3}, implies that $\O1a$ and
$\O1b$, taken together with $\O2c$, $\O2d$, and $\O3b$, indeed
provide a generating set. Moreover, a slight modification of Lemma
\ref{lem:O1} shows that any of the other three pairs of $\O1$ moves
in the statement of Theorem \ref{thm:othersets} may be used instead
of $\O1a$ and $\O1b$. Thus we see that all sets described in Theorem
\ref{thm:othersets} are indeed generating and obtain the ``if'' part
of the theorem. It remains to prove the ``only if'' part of Theorem
\ref{thm:othersets}, i.e., to show that other combinations of four
$\O1$ and $\O2$ moves, taken together with $\O3b$, do not result in
generating sets. We will proceed in three steps:

\begin{list}{}{\leftmargin=2.0cm\labelsep=0.5cm\labelwidth=1.6cm}
\item[\bf Step 1.] Prove that any such generating set should contain
at least two $\O1$ moves and to eliminate two remaining pairs
($\O1a$, $\O1d$) and ($\O1b$, $\O1c$) of $\O1$ moves.

\item[\bf Step 2.] Prove that any such generating set should contain
at least two $\O2$ moves and to eliminate pairs ($\O2a$,$\O2c$),
($\O2a$,$\O2d$), ($\O2b$,$\O2c$), and ($\O2b$,$\O2d$).

\item[\bf Step 3.] Eliminate the remaining pair ($\O2a$,$\O2b$).
\end{list}

The remainder of this section is dedicated to these three steps.
Step 1 is the simplest and is given by Lemma \ref{lem:Oe} below.
Step 2 is the most complicated; it is given by Corollaries
\ref{cor:O2acd} and \ref{cor:O2bcd}. Step 3 is relatively simple and
is given by Lemma \ref{lem:O2abO3b}.

To show that a certain set of Reidemeister moves is not generating,
we will construct an invariant of these moves which, however, is not
preserved under the set of all Reidemeister moves. The simplest
classical invariants of this type are the {\em writhe} $w$ and the
{\em winding number} $rot$ of the diagram. The winding number of the
diagram grows (respectively drops) by one under $\O1b$ and $\O1d$
(respectively $\O1a$ and $\O1c$). The writhe of the diagram grows
(respectively drops) by one under $\O1a$ and $\O1b$ (respectively
$\O1c$ and $\O1d$). Moves $\O2$ and $\O3$ do not change $w$ and
$rot$. These simple invariants suffice to deal with moves of type
one (see e.g. \cite{Oe}):

\begin{lem}[\cite{Oe}]\label{lem:Oe}
Any generating set of Reidemeister moves contains at least two $\O1$
moves. None of the two pairs ($\O1a$, $\O1d$) or ($\O1b$, $\O1c$),
taken together with all $\O2$ and $\O3$ moves, gives a generating
set.
\end{lem}

\begin{proof}
Indeed, both $\O1a$ and $\O1d$ preserve $w+rot$, so this pair (or
any of them separately) together with $\O2$ and $\O3$ moves cannot
generate all Reidemeister moves. The case of $\O1b$ and $\O1c$ is
obtained by the reversal of an orientation (of all components) of
the link.
\end{proof}

This concludes Step 1 of the proof. Let us proceed with Step 2.
 Here the situation is quite delicate, since the standard
algebraic/topological invariants, reasonably well behaved under
compositions, can not be applied. The reason can be explained on a
simple example: suppose that we want to show that $\O2d$ cannot be
obtained by a sequence of Reidemeister moves which includes $\O2c$.
Then our invariant should be preserved under $\O2c$ and distinguish
two tangles shown in Figure \ref{fig:composition}a. However, if we
compose them with a crossing, as shown in Figure
\ref{fig:composition}b, we may pass from one to another by $\O2c$.
Thus the invariant should not survive composition of tangles.

\begin{figure}[htb]
\centerline{\includegraphics[height=1.1in]{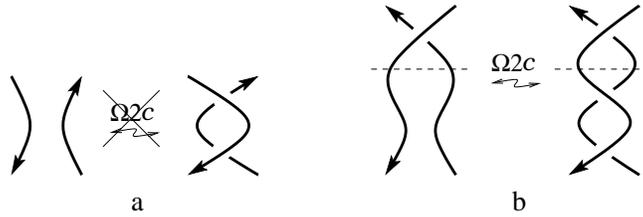}}
\caption{\label{fig:composition} Composition destroys inequivalence}
\end{figure}

Instead, we will use a certain notion of positivity, which is indeed
destroyed by such compositions. It is defined as follows. Let $D$ be
a $(2,2)$-tangle diagram with two oriented ordered components $D_1$,
$D_2$. Decorate all arcs of both components of $D$ with an integer
weight by the following rule. Start walking on $D_1$ along the
orientation. Assign zero to the initial arc. Each time when we pass
an overcrossing (we don't count undercrossings) with $D_2$, we add a
sign (the local writhe) of this overcrossing to the weight of the
previous arc. Now, start walking on $D_2$ along the orientation.
Again, assign zero to the initial arc. Each time when we pass an
undercrossing (now we don't count overcrossings) with $D_1$, we add
a sign of this undercrossing to the weight of the previous arc. See
Figure \ref{fig:weights1}a. Two simple examples are shown in Figure
\ref{fig:weights1}b,c.

\begin{figure}[htb]
\centerline{\includegraphics[width=5.0in]{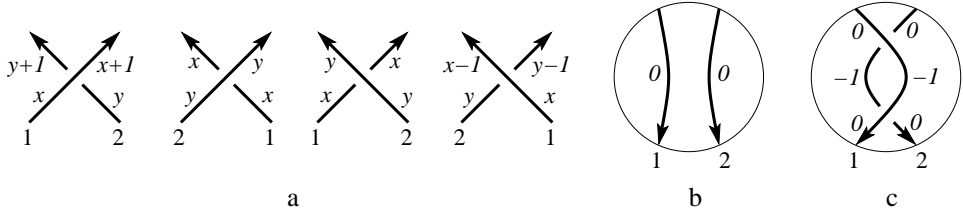}}
\caption{\label{fig:weights1} Weights of diagrams}
\end{figure}

We call a component {\em positively weighted}, if weights of all its
arcs are non-negative. E.g., both components of the (trivial) tangle
in Figure \ref{fig:weights1}b are positively weighted. None of the
components of a diagram in Figure \ref{fig:weights1}c are positively
weighted (since the weights of the middle arcs on both components are
$-1$). Behavior of positivity under Reidemeister moves is considered
in the next lemmas.

Denote by $S_b$ the set which consists of all $\O1$ moves and
$\O3b$.

\begin{lem}
Let $D$ be a $(2,2)$-tangle diagram with both positively weighted
components. Then both components of a diagram obtained from it by a
sequence of moves which belong to $S_b\cup\O2a$ are also positively
weighted.
\end{lem}

\begin{proof}
Indeed, an application of a first Reidemeister move does not change
this property since we count only intersections of two different
components. An application of $\O2a$ adds (or removes) two crossings
on each component in such a way, that walking along a component we
first meet a positive crossing and then the negative one, so the
weights of the middle arcs are either the same or larger than on the
surrounding arcs, see Figure \ref{fig:weights2}a. An application of
$\O3b$ preserves the weights since $\O3b$ involves only positive
crossings.
\end{proof}

\begin{figure}[htb]
\centerline{\includegraphics[width=4.7in]{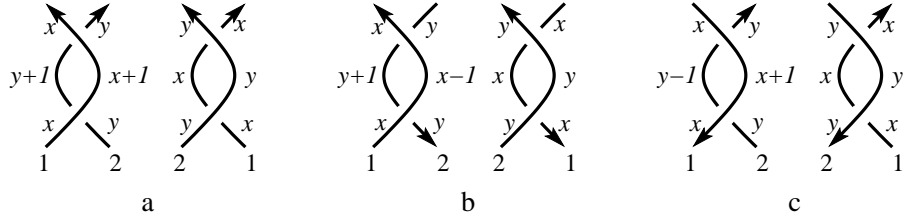}}
\caption{\label{fig:weights2} Weights and Reidemeister moves of type two}
\end{figure}

\begin{lem}
Let $D$ be a $(2,2)$-tangle diagram with a positively weighted
second component. Then any diagram obtained from it by $\O2c$ also
has a positively weighted second component.
\end{lem}

\begin{proof}
An application of $\O2c$ may add (or remove) two undercrossings
on $D_2$, but in such a way that we first meet a positive
undercrossing and then the negative one, so the weight of a middle
arc is larger than on the surrounding arcs, see Figure
\ref{fig:weights2}b.
\end{proof}

\begin{lem}
Let $D$ be a $(2,2)$-tangle diagram with a positively weighted first
component. Then any diagram obtained from it by $\O2d$ also has a
positively weighted first component.
\end{lem}

\begin{proof}
An application of $\O2d$ may add (or remove) two overcrossings on
$D_1$, but in such a way that we first meet a positive overcrossing
and then the negative one, so the weight of a middle arc is larger
than on the surrounding arcs, see Figure \ref{fig:weights2}c.
\end{proof}

Comparing  Figures \ref{fig:weights1}b and \ref{fig:weights1}c we
conclude
\begin{cor}\label{cor:O2acd}
None of the two sets $S_b\cup\O2a\cup\O2c$ and $S_b\cup\O2a\cup\O2d$
generates $\O2b$.
\end{cor}

The reversal of orientations (of both components) of the tangle in
the above construction gives
\begin{cor}\label{cor:O2bcd}
None of the two sets $S_b\cup\O2b\cup\O2c$ and $S_b\cup\O2b\cup\O2d$
generates $\O2a$.
\end{cor}

\begin{rem}
In \cite[Theorem 5.4]{Tu} (and later \cite{Wu}) the set
$S_b\cup\O2a\cup\O2c$ is considered as a generating set. Fortunately
(V.~Turaev, personal communication), an addition of $\O2d$ does not
change the proof of the invariance in \cite[Theorem 5.4]{Tu}.
\end{rem}

Note that the above corollaries imply that any generating set $S$
which contains only one move, $\O3b$, of type three, should contain
at least two $\O2$ moves. This concludes Step 2 of the proof.

Since at the same time such a set $S$ should contain at least two
$\O1$ moves by Lemma \ref{lem:Oe}, we conclude that if $S$ consists
of five moves, there should be exactly two $\O2$ moves and two $\O1$
moves. This simple observation allows us to eliminate the last
remaining case:

\begin{lem}\label{lem:O2abO3b}
Let $S$ be a set which consists of two $\O1$ moves, $\O2a$, $\O2b$,
and $\O3b$. Then $S$ is not generating.
\end{lem}

\begin{proof}
Given a link diagram, smooth all double points of the diagram
respecting the orientation, as illustrated in Figure
\ref{fig:smooth}.

\begin{figure}[htb]
\centerline{\includegraphics[height=0.72in]{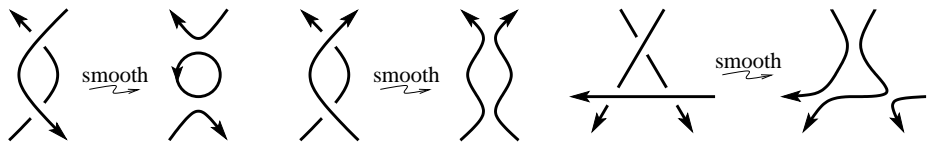}}
\caption{\label{fig:smooth} Smoothing the diagram respecting the
orientation}
\end{figure}

Count the numbers $C^-$ and $C^+$ of clockwise and counter-clockwise
oriented circles of the smoothed diagram, respectively. Note that
$\O2a$, $\O2b$, and $\O3b$ preserve an isotopy class of the smoothed
diagram, thus preserve both $C^+$ and $C^-$. On the other hand,
$\O1b$ and $\O1d$ add one to $C^+$, and $\O1a$, $\O1c$ add one to
$C^-$. Thus if $S$ contains $\O1a$ and $\O1c$, all moves of $S$
preserve $C^+$. The case of $\O1b$ and $\O1d$ is obtained by the
reversal of an orientation (of all components) of the link. If $S$
contains $\O1a$ and $\O1b$, all moves of $S$ preserve $C^++C^--w$.
Similarly, if $S$ contains $\O1c$ and $\O1d$, all moves of $S$
preserve $C^++C^-+w$. In all the above cases, moves from $S$ can not
generate $\O2c$, $\O2d$, since each of $\O2c$ and $\O2d$ may change
$C^+$ as well as $C^++C^-\pm w$ (while preserving $w$ and
$C^+-C^-=rot$).
\end{proof}

This concludes the proof of Theorem \ref{thm:othersets}.


\end{document}